\documentclass[11pt,a4paper,reqno]{amsart}

\usepackage{amsmath}
\usepackage{amsfonts}
\usepackage{a4wide}
\usepackage{amssymb}
\usepackage{amsthm}

\theoremstyle{plain}

\newtheorem{teo}{Theorem}[section]

\newtheorem{cor}[teo]{Corollary}
\newtheorem{ackn}{Acknowledgments\!}

\theoremstyle{definition}

\theoremstyle{remark}

\newtheorem{rem}[teo]{Remark}

\numberwithin{equation}{section}

\def\SS{{{\mathbb S}}}

\def\RR{{\mathbb R}}

\def\Ric{{\mathrm {Ric}}}

\usepackage[usenames]{color}
\definecolor{red}{rgb}{1.0,0.0,0.0}

\definecolor{blu}{rgb}{0.0,0.0,1.0}

\title[On the global structure of conformal gradient solitons
with $\Ric\geq 0$]{On the global 
structure of conformal gradient solitons with nonnegative Ricci tensor}
\date{\today}

\author[Giovanni Catino]{Giovanni Catino}
\address[Giovanni Catino]{Dipartimento di Matematica, Politecnico di Milano, Piazza Leonardo da Vinci 32, Milano, Italy, 20133}
\email[G. Catino]{giovanni.catino@polimi.it}

\author[Carlo Mantegazza]{Carlo Mantegazza}
\address[Carlo Mantegazza]{Scuola Normale Superiore di Pisa, Piazza dei Cavalieri 7, Pisa, Italy, 56126}
\email[C. Mantegazza]{c.mantegazza@sns.it}

\author[Lorenzo Mazzieri]{Lorenzo Mazzieri}
\address[Lorenzo Mazzieri]{Scuola Normale Superiore di Pisa, Piazza dei Cavalieri 7, Pisa, Italy, 56126}
\email[L. Mazzieri]{l.mazzieri@sns.it}

\begin{document}

\begin{abstract} In this paper we prove that any complete
  conformal gradient soliton with nonnegative Ricci tensor is either isometric to a direct product $\RR\times N^{n-1}$, or globally conformally equivalent to the Euclidean space $\RR^{n}$ or to the round sphere $\SS^{n}$. In particular, we  show that any complete, noncompact, gradient Yamabe--type soliton with positive Ricci tensor is rotationally symmetric, whenever the potential function is nonconstant.
\end{abstract}

\maketitle

\begin{center}


\medskip

\centerline{\bf AMS subject classification:  53A30, 53C24, 53C25, 53C44}

\end{center}

\section{Introduction}

A connected, complete Riemannian manifold $(M^{n},g)$ is called a {\em
  conformal gradient soliton} if there exists a {\em nonconstant} smooth
function $f$, called {\em potential} of the soliton, such that
\begin{equation*}
\nabla^{2} f \, = \, \varphi\, g \,,
\end{equation*}
for some function $\varphi:M^n\to\RR$. Tracing this equation with the metric $g$, we see
immediately that the function $\varphi$ must coincide with $\Delta f/n$.
Hence, an equivalent characterization of conformal gradient solitons is given by the equation
\begin{equation}\label{confsol}
\nabla^{2} f \, = \, \frac{\Delta f}{n}\, g \,.
\end{equation}

In this note we are going to fully detail a remark of Petersen and
Wylie~\cite[Remark A.3]{pw} about the classification of these
solitons, moreover, we revisit a result of Tashiro~\cite{tashiro}, who
first studied their global structure.

Complete Riemannian manifolds admitting a vector field $\nabla f$
satisfying equation~\eqref{confsol} were studied by many authors in the late
60's. Solutions to equation~\eqref{confsol}
have also been considered in a work by Cheeger and
Colding~\cite{cheegcold}, where the authors give a characterization of
warped product manifolds. In particular, they observe that, in the
complement of the critical points of $f$, any conformal gradient
soliton is isometric to a warped product on some open interval. Taking
advantage of this, we will be able to drastically simplify the proof
of the classification result for conformal gradient solitons given by
Tashiro. Moreover, we further characterize conformal gradient solitons
with nonnegative Ricci tensor, in the spirit of some recent works
about the classification of Einstein--like structures, such as
gradient Ricci solitons and quasi--Einstein manifolds.

Our main result reads:

\begin{teo}\label{teo1} Let $(M^{n},g)$ be a complete conformal gradient soliton and let $f$ be a potential function for it. Then, any regular level set $\Sigma$ of $f$ admits a maximal open neighborhood $U\subset M^n$ on which $f$ only depends on the signed distance $r$ to the hypersurface $\Sigma$. In addiction, the potential function $f$ can be chosen in such a way that the metric $g$ takes the form 
\begin{equation*}
g \, = \, dr^2 \,+ \,(f'(r))^{2}\, g^{\Sigma}\quad {\hbox{on $U$}} ,
\end{equation*}  
where $g^{\Sigma}$ is the metric induced by $g$ on $\Sigma$. As a consequence, $f$ has at most two critical points on $M^n$ and we have the following cases:
\begin{itemize}
\item[(1)] If $f$ has no critical points, then $(M^{n},g)$ is globally conformally equivalent to a direct product $I\times N^{n-1}$ of some interval $I=(t_{*},t^{*})\subseteq \RR$ with a $(n-1)$--dimensional complete Riemannian manifold $(N^{n-1},g^{N})$. More precisely, the metric takes the form
$$
g \, = \, u^{2}(t)\, \big(dt^{2}+g^{N}\big) \, ,
$$
where $u:(t_{*},t^{*})\rightarrow \RR$ is some positive smooth
function. In this case, if $(M^n,g)$ is also locally
conformally flat, it is well known that $(N^{n-1},g^N)$ must have constant curvature.
\smallskip
\item[(1')] If, in addition, the Ricci tensor of $(M^n,g)$ is nonnegative, then $(M^n,g)$ is {\em isometric} to a
direct product $\RR\times N^{n-1}$, where $(N^{n-1},g^N)$ has nonnegative
Ricci tensor. In this case, if $(M^n,g)$ is also locally
conformally flat, then either $(M^n,g)$ is flat or it is a direct product of $\RR$ with a quotient of the round sphere $\SS^{n-1}$.

\smallskip

\item[(2)] If $f$ has only one critical point $O\in M^n$, then 
$(M^{n},g)$ is globally conformally equivalent to the interior of an Euclidean ball of radius $t^{*}\in(0,+\infty]$.
More precisely, on $M^{n}\setminus~\{O\}$, the metric takes the form
$$
g \, = \, v^{2}(t)\, \big(dt^{2}+t^{2}g^{\SS^{n-1}}\big) \,,
$$
where $v:(0,t^{*})\rightarrow \RR$ is some positive smooth function. In particular $(M^{n},g)$ is complete, noncompact and rotationally symmetric.
\smallskip
\item[(2')] If, in addition, the Ricci tensor of $(M^n,g)$ is nonnegative, then $(M^n,g)$ is globally conformally equivalent to $\RR^{n}$.

\smallskip

\item[(3)] If the function $f$ has two critical points $N,S \in M^n$, then $(M^{n},g)$ is globally conformally equivalent to $\SS^{n}$.  More precisely, on $M^{n}\setminus \{N,S\}$, the metric takes the form
$$
g \, = \, w^{2}(t)\, \big(dt^{2}+\sin^{2}(t)\,g^{\SS^{n-1}}\big) \,,
$$
where $w:(0,\pi)\rightarrow \RR$ is some smooth positive function. In particular $(M^{n},g)$ is compact and rotationally symmetric.
\end{itemize}  
\end{teo}

In Section~\ref{s_proof} we will prove Theorem~\ref{teo1}, whereas in
Section~\ref{s_yamabe} we will focus our attention on the
classification of gradient Yamabe solitons. These are conformal
gradient solitons satisfying the equation
\begin{equation*}
\nabla^{2}f \, = \, (R-\lambda)\,g\,,
\end{equation*}
for some constant $\lambda$. We will show that any complete,
noncompact, gradient Yamabe soliton with nonnegative Ricci tensor
either has constant scalar curvature, or it splits isometrically as a
direct product $\RR\times N^{n-1}$, or it is rotationally symmetric
and globally conformally equivalent to $\RR^{n}$ 
(see Theorem~\ref{teoY} and Theorem~\ref{teokY} for the generalization
to the case of gradient $k$--Yamabe solitons).

\medskip

\section{Proof of Theorem~\ref{teo1}}\label{s_proof}

Let $\Sigma$ be a regular level set of the function $f:M^n\to\RR$,
i.e. $|\nabla f|\neq 0$ on $\Sigma$, which exists by Sard's Theorem and the fact that $f$ is nonconstant in our definition. We have that $|\nabla f|$ has to be
constant on $\Sigma$. Indeed, for all $X\in T_{p}\Sigma$
$$
\nabla_{X} |\nabla f|^{2} \,=\, 2 \,\nabla^{2} f (\nabla f, X) =
\frac{2\,\Delta f}{n} \,g(\nabla f,X) \, = \, 0\,.
$$
From this we deduce that, in a neighborhood $U$ of $\Sigma$ not
containing any critical point of $f$, such potential function only depends
on the signed distance $r$ to the hypersurface $\Sigma$. In particular
$df=f' dr$. Moreover, if $\theta=(\theta^{1}\,\ldots,\theta^{n-1})$ are coordinates {\em adapted} to the hypersurface $\Sigma$, we get
$$
\nabla^{2}f \,=\, \nabla df \,=\, f'' dr\otimes dr + f' \nabla^2 
r=\, f'' dr\otimes dr + \frac{f'}{2} \, \partial_r g_{ij} \,d\theta^i\otimes d\theta^j\,,
$$
as
$$
\Gamma_{rr}^r=\Gamma_{rr}^k=\Gamma_{ir}^r=0\,,\qquad
\Gamma_{ij}^r=- \frac{1}{2} \,\partial_r g_{ij}\,,\qquad
\Gamma_{ir}^k= \frac{1}{2} \, g^{ks}\partial_r g_{is}\,.
$$
On the other hand, using equation~\eqref{confsol}, we have
$$
\nabla^{2} f \, = \, \frac{\Delta f}{n}\, g \,= \, \frac{\Delta
  f}{n}\,(\,dr \otimes dr + g_{ij}\, d\theta^{i} \otimes
d\theta^{j}\,)\,,
$$
thus, $\Delta f = n f''$ and $g_{ij} \Delta f  = \frac{n}{2} f' \,\partial_{r}
g_{ij}$. These equations imply the family of ODE's
$$
f''(r) \,g_{ij}(r,\theta)  \,=\, \frac{f'(r)}{2} \,\partial_{r}
g_{ij}(r,\theta)\,.
$$
Since $f'(0)\not=0$ (otherwise $\Sigma$ is not a regular level set of $f$) we can integrate these equations obtaining
$$
g_{ij}(r,\theta) \, = \, \big(f'(r)/f'(0)\big)^{2} g_{ij}(0,\theta)\,.
$$
Therefore, in $U$ the metric takes the form
$$
g \, = \, dr \otimes dr \,+ \,\big(f'(r)/f'(0)\big)^{2}\, \sigma_{ij}(\theta)\,d\theta^{i} \otimes d\theta^{j}\,,
$$
where $g^{\Sigma}_{ij}(\theta)=g_{ij}(0,\theta)$ is the metric induced by $g$ on $\Sigma$. We notice that, since $f=f(r)$, then the width of the neighborhood $U$ is uniform with respect to the points of $\Sigma$, namely we can assume $U=\{r_{*}<r<r^{*}\}$, for some maximal $r_{*}\in[-\infty,0)$ and $r^{*}\in(0,\infty]$. Moreover, by the scalar invariance of equation~\eqref{confsol}, we can assume that $f'(0)=1$, possibly changing the function $f$. Hence, in $U$, the metric can be written as
\begin{equation}\label{metric}
g \, = \, dr \otimes dr \,+ \,(f')^{2}\, g^{\Sigma}\,,
\end{equation}
where $g^{\Sigma}$ denotes the induced metric on the level set $\Sigma$. Moreover, the Ricci tensor and the scalar curvature of the metric $g$ take the form (see~\cite[Proposition~9.106]{besse}) 
\begin{equation}\label{ricci}
\Ric_{g} = -(n-1) \frac{f'''}{f'} dr\otimes dr + \Ric^{\Sigma} - \big((n-2)(f'')^{2}) + f'\,f'''\big)\,g^{\Sigma}\,,
\end{equation}
\begin{equation}\label{scalar}
R_{g} = -2(n-1)\frac{f'''}{f'} + \frac{R^\Sigma-(n-1)(n-2)(f'')^{2}}{(f')^{2}}\,.
\end{equation} 

\medskip

{\bf Case 1: $f$ has no critical points.} Since $(M^{n},g)$ is complete, the width of the maximal neighborhood $U$ is unbounded in both the negative and the positive direction of the signed distance $r$ (i.e., $r_{*}=-\infty$ and $r^{*}=+\infty$). To complete the proof, it is sufficient to set
$$
t(r) \,=\, \int_{0}^{r}\frac{1}{f'(z)}\,dz\,.
$$
For $r\in(-\infty,+\infty)$, we have $t\in(t_{*},t^{*})$, where $t_{*}=\lim_{r \rightarrow r_{*}} t(r)\in[-\infty,0)$ and $t^{*}=\lim_{r \rightarrow r^{*}} t(r)\in(0,+\infty]$. Moreover, $t'(r)\neq 0$ and $r$ can be viewed as a function of $t$ by inverting the expression above. From~\eqref{metric}, the metric takes the form
$$
g \,=\, u(t)^{2}\big(\,dt^{2}+g^{\Sigma}\,\big)\,,
$$
where $u(t)=f'(r(t))$.

\medskip

{\bf Case 1': $f$ has no critical points and $\Ric\geq 0$.} From formula~\eqref{ricci} and the fact that $g$ has nonnegative Ricci tensor, one has
$$
0\,\leq\, \Ric_{g}(\partial_r, \partial_r) \,=\, -(n-1)\frac{f'''}{f'}\,.
$$
Hence, $f'$ is a concave function defined on the whole real line that can
never be zero. Thus, it must be constant, that is $f'\equiv 1$, according to
our choice of $f'(0)$. This implies that $(M^{n},g)$ is isometric to the
direct product $\RR\times N^{n-1}$ of the real line with a
$(n-1)$--dimensional complete Riemannian manifold with nonnegative Ricci tensor.

\begin{rem}\label{remark}  We notice that, in this case, the Ricci tensor has a zero eigenvalue at every point. Hence, there are no examples of such manifolds for which the Ricci tensor is positive definite at some point.
\end{rem}

\medskip

 {\bf Case 2: $f$ has only one critical point $O\in M^n$.} In this case, since $(M^{n},g)$ is complete, we can assume that the width of the neighborhood $U$ is unbounded in the positive direction of the signed distance (i.e., $r_{*}>-\infty$ and $r^{*}=+\infty$) and $f'\rightarrow 0$, as $r\rightarrow r_{*}$. By formula~\eqref{ricci} and the smoothness of the metric $g$, we have that $f'''/f'$ is bounded, as $r\rightarrow r_{*}$. Hence, from~\eqref{scalar}, we deduce that
$$
R^{\Sigma}-(n-1)(n-2)(f'')^{2}\longrightarrow 0\,,\quad\quad\hbox{as}\,\, r\rightarrow r_{*}\,.
$$
In particular $R^{\Sigma}$ is nonnegative and constant along $\Sigma$. Moreover, it is easy to see that
\begin{equation}\label{hopital}
\lim_{r\to r_{*}} \,\frac{f'(r)}{r-r_{*}} \, = \lim_{r\to r_*}
f''(r)\, = \bigl(R^{\Sigma}/(n-1)(n-2)\bigr)^{1/2}\,.
\end{equation} 

To conclude the proof of Case~2 of Theorem~\ref{teo1}, it remains to
show that the induced metric $g^{\Sigma}$ on the level set $\Sigma$
is proportional to the round metric $g^{\SS^{n-1}}$ of the
$(n-1)$--dimensional sphere. This follows from the
elementary fact that, infinitesimally, the metric $g$ is approximately
Euclidean near $O$. 
Indeed, the standard expansion of the metric $g$ around $O$, written
in normal coordinates $(x^1, \cdots, x^n)$, gives
\begin{eqnarray*} 
g &=& (\delta_{ij}+ \eta_{ij}(x))\, dx^{i}\otimes dx^{j} \\
&=& g^{\mathbb{R}^{n}}+ \eta_{ij}\, dx^{i}\otimes dx^{j} \,,
\end{eqnarray*} 
where $\eta_{ij}=\mathcal{O}(|x|^{2})$. Passing to Riemannian polar
coordinates, we write $x^{i} = s
\,\phi^{i}(\vartheta^{1},\ldots,\vartheta^{n-1)})$, with
$s=r-r_{*}\in(0,+\infty)$ and $(\vartheta^{1},\ldots,\vartheta^{n-1})$
being local coordinates on $\mathbb{S}^{n-1}$. Notice that
$|\phi^{1}|^{2}+\dots+|\phi^{n}|^{2}=1$ and $|x|=s$. Thus, one has
\begin{eqnarray*}
g &=& ds\otimes ds + \big( s^{2}
{g}^{\mathbb{S}^{n-1}}_{\alpha\beta}+\,s^{2}\eta_{ij}\frac{\partial
  \phi^{i}}{\partial \vartheta^{\alpha}}\frac{\partial
  \phi^{j}}{\partial \vartheta^{\beta}}\big) \,d\vartheta^{\alpha}
\otimes d\vartheta^{\beta}\,,
\end{eqnarray*} 
with $\eta_{ij}=\mathcal{O}(s^{2})$. Comparing with
expression~\eqref{metric}, we see that, for $s\in(0,+\infty)$, we have
$$
f'(s+r_{*})^{2} g^{\Sigma} \,= \, s^{2} {g}_{\mathbb{S}^{n-1}} +
s^{2}\eta_{ij}\frac{\partial \phi^{i}}{\partial
  \vartheta^{\alpha}}\frac{\partial \phi^{j}}{\partial
  \vartheta^{\beta}} \,d\vartheta^{\alpha} \otimes
d\vartheta^{\beta}\,.
$$
Now, combining the fact that $\eta_{ij}=\mathcal{O}(s^{2})$
with formula~\eqref{hopital}, if we take the limit as $s\to 0$ (which
means $r\to r_{*}$) we obtain $R^{\Sigma}>0$ and
$$
g^{\Sigma} \,=\, c^{2}\,{g}_{\mathbb{S}^{n-1}} \,,
$$ 
with $c^{2}=(n-1)(n-2)/R^{\Sigma}$. Therefore, on $M^n\setminus\{O\}$, we have 
$$
g = ds^{2} + (c\,f'(s+r_{*}))^{2} \,{g}_{\mathbb{S}^{n-1}}\,.
$$
This proves that $g$ is rotationally
symmetric. To complete the proof, we set
\begin{equation}\label{eqqq7}
t(s) \, = \, \exp\Big(\,\frac{1}{c}\,\int_{-r_{*}}^{s}\frac{1}{f'(z+r_{*})}\,dz\,\Big)\,.
\end{equation}
For $s\in(0,+\infty)$, we have $t\in(0,t^{*})$, where $t^{*}=\lim_{s
  \to +\infty} t(s)\in(0,+\infty]$. Notice that $t'(s)\neq 0$, hence,
the coordinate $s$ can be viewed as a function of $t$ by inverting the
expression above. Moreover,
$$
\frac{dt}{t} \,=\, \frac{ds}{c\,f'(s+r_{*})}
$$
and the metric $g$ can be expressed as in the statement, namely
$$
g\,=\, v(t)^{2}\big(\,dt^{2}+t^{2}\,g^{\SS^{n-1}}\,\big)\,,
$$
where $v(t)=c\,f'(s(t)+r_{*})/t$.

\medskip

{\bf Case 2': $f$ has one critical point $O\in M^n$ and $\Ric\geq
  0$.} Like in the Case 1' above, as $\Ric\geq0$, we have that $z\mapsto f'(z+r_{*})$
is a concave function. In particular, it is definitely bounded above by some linear function, as $z\rightarrow +\infty$. 
Then, by the very definition of $t^*$ in equation~\eqref{eqqq7}, it
follows that $t^*=+\infty$. This clearly implies that $(M^{n},g)$ is globally conformally
equivalent to $\RR^{n}$ and rotationally symmetric.

\medskip

{\bf  Case 3: $f$ has two critical points $N,S\in M^n$.} 
We assume that the width of the neighborhood $U$ is bounded in
both the negative and the positive directions of the signed distance
(i.e., $r_{*}>-\infty$ and $r^{*}<+\infty$). In particular $(M^{n},g)$
is compact (it ``closes'' at the points $N$ and $S$) and there cannot
be other critical points around. The same argument used in the proof of Case~2 implies at once the
rotational symmetry of $g$. 
Namely, on $M^n\setminus\{N,S\}$, we have 
$$
g = ds^{2} + (c\,f'(s+r_{*}))^{2} \,{g}_{\mathbb{S}^{n-1}}\,,
$$
where $c^{2}=(n-1)(n-2)/R^{\Sigma}$. To complete the proof of Case 3, we set
$$
t(s) \, = \, 2\arctan\exp\Big(\,\frac{1}{c}\,\int_{-r_{*}}^{s}\frac{1}{f'(z+r_{*})}\,dz\,\Big)\,.
$$
For $s\in(0,r^{*}-r_{*})$, we have $t\in(0,\pi)$ and $t'(s)\neq 0$,
hence, $s$ can be viewed as a function of $t$ by inverting the
expression above. Moreover,
$$
\frac{dt}{\sin(t)} \,=\, \frac{ds}{c\,f'(s+r_{*})}
$$
and the metric $g$ can be expressed as in the statement, namely
$$
g\,=\, w(t)^{2}\big(\,dt^{2}+\sin^{2}(t)\,g^{\SS^{n-1}}\,\big)
$$
where $w(t)=c\,f'(s(t)+r_{*})/\sin(t)$.

\medskip

This completes the proof of Theorem~\ref{teo1}.

\medskip

\section{Classification of Yamabe--Type Solitons with Nonnegative Ricci
  Tensor}\label{s_yamabe}

Let $(M^{n},g)$, $n\geq 3$, be a complete Riemannian manifold verifying
\begin{equation}\label{sol}
\nabla^{2} f \, = \, \varphi\, g \,,
\end{equation}
for some smooth functions $f$ and $\varphi$ on $M^{n}$. When the potential function $f$ is nonconstant, then, according to our definition, $(M^{n},g)$ is a {\em conformal gradient soliton} and Theorem~\ref{teo1} applies.

We first notice that, by taking the divergence of this equation, we have
$$
\nabla _{i} \varphi \, = \, \Delta \nabla_{i} f \, =
\,\nabla_{j}\nabla_{i}\nabla^{j} f \, 
= \, \nabla_{i} \Delta f + R_{ij}\nabla^{j}f\,,
$$
where we interchanged the covariant derivatives. Now, using
the fact that $\Delta f=n\,\varphi$, we obtain the following identity 
\begin{equation}\label{eq2}
(n-1) \, \nabla_{i} \varphi \,=\, - R_{ij} \nabla^{j} f\,.
\end{equation}

We will discuss now some geometric applications of Theorem~\ref{teo1} to {\em gradient Yamabe solitons} and {\em gradient $k$--Yamabe solitons}.

\medskip 

\subsection{Gradient Yamabe Solitons}
A Riemannian manifold $(M^{n},g)$ is called a {\em gradient Yamabe
  soliton} if it satisfies equation~\eqref{sol} with $\varphi \, =
\, R-\lambda$ for some constant $\lambda\in\RR$, i.e., there exists
a smooth function $f$ (notice that here we are not excluding the case of a constant  $f$) such that
\begin{equation}\label{Ysol}
\nabla^{2}f \, = \, (R-\lambda)\,g\,.
\end{equation}
If $\lambda=0$, $\lambda>0$ or $\lambda<0$, then the soliton is called {\em
  steady}, {\em shrinking} or {\em expanding}, respectively. We recall that gradient Yamabe solitons are self--similar solutions to the Yamabe flow 
$$
\frac{\partial}{\partial t} \,g \, = \, -R\,g\,.
$$
This flow was first introduced by Hamilton and we refer the reader to~\cite{dasksesum} and the references therein for further details on this subject.

Notice that any Riemannian manifold with constant scalar curvature
moves by the Yamabe flow only by dilations. Hence, it is trivially a
self--similar solution and a gradient Yamabe soliton with $R=\lambda$ and $f$ constant. Thus, according to our definitions, only gradient Yamabe solitons with nonconstant potential function $f$ can be viewed as conformal gradient solitons. 

On the other hand, it is well known (see, for
instance~\cite[Proposition~B.16]{chowluni}) 
that any compact gradient Yamabe soliton has constant
scalar curvature $R=\lambda$. For the sake of completeness, we report here the proof.
\begin{teo}\label{teoYcpt} Any compact gradient Yamabe soliton has 
constant scalar curvature $R=\lambda$. Moreover, the potential function $f$ is constant.
\end{teo}
\begin{proof}
Contracting equation~\eqref{Ysol} with the Ricci tensor and
integrating over $M^n$, we obtain
$$
\int_{M^n}(R-\lambda)R\,dV_{g} \,= \, \int_{M^n}R_{ij}\,\nabla^{ij}
f\,dV_{g} \,=\, - \int_{M^n}\nabla_{i}R_{ij}\,\nabla_{j}f\,dV_{g}\,=\,
-\frac{1}{2}\int_{M^n}\langle\nabla R,\nabla f\rangle\,dV_{g}\,,
$$
where in the last equality we have used Schur's lemma
$2\,\hbox{div}(\Ric) = dR$. Moreover, from equation~\eqref{Ysol}, one
has that $\Delta f=n(R-\lambda)$. Hence, it follows that
$\lambda=\tfrac{1}{Vol(M^n)}\int_{M^n} R\,dV_g$ and, from the previous
computation, we get
$$
\int_{M^n}(R-\lambda)^{2}\,dV_{g} \, = \,
-\frac{1}{2}\int_{M^n}\langle\nabla R,\nabla f\rangle\,dV_{g} \, = \,
\frac{n}{2}\int_{M^n} (R-\lambda)^{2}\,dV_{g}\,.
$$
Since $n\geq 3$, this implies that $R$ coincides with the constant $\lambda$. As an immediate consequence, by the relation~\eqref{Ysol}, we have
that $\Delta f$ is zero. Since $M^n$ is compact, the function $f$ is constant as well.
\end{proof}

When the function $f$ is constant, obviously the scalar curvature of the gradient Yamabe soliton is constant as well. In such a case, Theorem~\ref{teo1} does not apply. In the sequel we will always assume that relation~\eqref{Ysol} is satisfied by some nonconstant function $f$, so we will deal only with noncompact gradient Yamabe solitons. In this case, as an immediate application of Theorem~\ref{teo1}, we can prove the following global result.

\begin{teo}\label{teoY} Let $(M^{n},g)$ be a complete, noncompact,
  gradient Yamabe soliton with nonnegative Ricci tensor and
  nonconstant potential function $f$. Then, we have the following two cases:
\begin{itemize}  
 \item[(1)] either $(M^{n},g)$ is a direct product
  $\RR\times N^{n-1}$ where $(N^{n-1},g^N)$ is 
an $(n-1)$--dimensional complete Riemannian manifold with nonnegative Ricci tensor. If in addition $(M^{n},g)$ is locally conformally flat, then either it is flat or 
the manifold $(N^{n-1},g^N)$ is a quotient of the round sphere $\SS^{n-1}$;
\smallskip
\item[(2)] or $(M^{n},g)$ is rotationally symmetric and globally
  conformally equivalent to $\RR^n$. More precisely, there exists a
  point $O\in M^n$ such that on $M^{n}\setminus \{O\}$,
  the metric has the form
$$
g \, = \, v^{2}(t)\, \big(dt^{2}+t^{2}g^{\SS^{n-1}}\big) \,,
$$
where $v:\RR^+\to \RR$ is some positive smooth function.
\end{itemize}
\end{teo}

From Remark~\ref{remark}, it is now easy to deduce the following
corollary.

\begin{cor}\label{corY} Let $(M^{n},g)$ be a complete, noncompact,
  gradient Yamabe soliton with nonnegative Ricci tensor and
  nonconstant potential function $f$. If the Ricci tensor is positive definite 
  at some point, then $(M^{n},g)$ is rotationally symmetric and globally
  conformally equivalent to $\RR^n$, in particular, it is locally conformally flat.
\end{cor}

It was proved by P. Daskalopoulos and N. Sesum in~\cite{dasksesum},
that any complete, noncompact, locally conformally flat, gradient
Yamabe soliton with positive sectional curvature 
has to be globally conformally equivalent to
$\RR^{n}$. Corollary~\ref{corY} shows that one can remove the assumption
of local conformal flatness and relax the hypothesis on the
sectional curvature. In~\cite{dasksesum} the authors also provide a 
complete classification of all rotationally symmetric gradient Yamabe
solitons in the steady, shrinking and expanding cases.

\medskip

\subsection{Gradient $k$--Yamabe Solitons}

A Riemannian manifold $(M^{n},g)$ is called a {\em gradient
  $k$--Yamabe soliton} if satisfies equation~\eqref{sol} with
$\varphi \, = \, 2(n-1)(\sigma_{k}-\lambda)$ for some constant
$\lambda\in\RR$, where $\sigma_{k}$ denotes the
$\sigma_{k}$--curvature of $g$. We recall that, if we denote by
$\mu_{1}, \ldots, \mu_{n}$ the eigenvalues of the symmetric
endomorphism $g^{-1}A$, where $A$ is the Schouten tensor defined by
$$
A \,= \, \tfrac{1}{ n-2} \,\big( \, \Ric -  \tfrac{1}{2(n-1)} \,R \,g \, \big) \, , 
$$
then the $\sigma_k$--curvature of $g$ is defined as the $k$--th
symmetric elementary function of $\mu_{1},\ldots,\mu_{n}$, namely
\begin{eqnarray*}
\sigma_{k}\,=\,\sigma_k(g^{-1} A) \, = \, \sum_{i_1\, <\,\ldots\,< \,
  i_k}\mu_{i_i}\cdot \, \ldots \, \cdot \mu_{i_k} \,\, \quad \hbox{for
  $1\leq k \leq n$}\,.
\end{eqnarray*}
Notice that $\sigma_{1}=\tfrac{1}{2(n-1)}R$, so gradient $1$--Yamabe solitons simply correspond to gradient Yamabe solitons. The structure equation takes the form
\begin{equation}\label{kYsol}
\nabla^{2}f \, = \, 2(n-1)(\sigma_{k}-\lambda)\,g\,,
\end{equation}
for some constant $\lambda\in\RR$. As usual, if $\lambda=0$,
$\lambda>0$ or $\lambda<0$, then $g$ is called {\em steady}, {\em shrinking} or {\em expanding}, respectively. 

Again, we observe that only gradient $k$--Yamabe solitons with nonconstant potential function $f$ can be viewed as conformal gradient solitons.

We have seen that, for $k=1$, compact, gradient
Yamabe solitons have constant scalar curvature. By means of a generalized
Kazdan--Warner identity, for any $k\geq 2$, we can prove the following analogue of Theorem~\ref{teoYcpt}.

\begin{teo} Any compact, gradient $k$--Yamabe soliton with nonnegative Ricci tensor has constant $\sigma_{k}$--curvature $\sigma_{k}=\lambda$. Moreover, the potential function $f$ is constant.
\end{teo}
\begin{proof} Let us suppose, by contradiction, that $\sigma_{k}$ is
  nonconstant. Then $f$ cannot be constant, since $\Delta f =2n(n-1)\,(\sigma_{k}-\lambda)$. Hence,
  we can apply Theorem~\ref{teo1}, obtaining that $(M^{n},g)$ is
  globally conformally equivalent to $\SS^{n}$, in particular, $g$ is
  locally conformally flat. It was proved in the proof in~\cite{han1} that, on a
  compact, locally conformally flat, Riemannian manifold, one has 
$$
\int_{M^n}\langle X,\nabla\sigma_{k}\rangle\,dV_{g} \,=\,0 \,,
$$
for every conformal Killing vector field $X$ on $(M^{n},g)$. For
$k=1$, this obstruction corresponds to the well known Kazdan--Warner
identity, which holds on any compact Riemannian manifold (i.e.,
without assuming the locally conformally flatness,
see~\cite{bourgezin}). From the structure equation~\eqref{kYsol}, we
know that $\nabla f$ is a conformal Killing vector field, hence, it
follows that
$$
\int_{M^n}\langle \nabla f,\nabla\sigma_{k}\rangle\,dV_{g} \,=\,0 \,.
$$
Now, contracting the identity~\eqref{eq2} with $\nabla f$, and integrating over $M^n$, we obtain
$$
0\,=\, \int_{M^n}\langle \nabla f,\nabla\sigma_{k}\rangle\,dV_{g} \,=\,
-\frac{1}{2(n-1)^{2}} \int_{M^n} \Ric(\nabla f,\nabla f)\,dV_{g} \,.
$$  
From the fact that $g$ has nonnegative Ricci tensor, we obtain that
$\Ric(\nabla f,\nabla f)=0$ everywhere. Then, by equation~\eqref{eq2}, we get $\langle \nabla f,\nabla
\sigma_{k}\rangle =0$. Since $g$ is rotationally symmetric, we have
that $\sigma_{k}$ is constant on the regular level sets of $f$. Hence,
the condition $\langle \nabla f,\nabla \sigma_{k}\rangle =0$ is sufficient
to conclude that $\sigma_{k}$ is constant. This implies that $\Delta f$ is constant. Since $M^n$ is compact, the only possibility is that $f$ is constant and $\sigma_{k}=\lambda$.
\end{proof}

\begin{rem} The same result holds if one considers a generalized $k$--Yamabe soliton structure~\eqref{sol} with $\varphi = \psi (\sigma_{k})$,
for every $k\geq 1$ and every strictly monotone function
$\psi:\RR\to\RR$. For instance, in~\cite{guanwang} the authors consider a
fully nonlinear conformal flow with velocity
$\varphi=\log(\sigma_{k})-\lambda$, with $\sigma_{k}>0$.
\end{rem}

Again, when the function $f$ is constant, obviously the $\sigma_{k}$--curvature of the gradient $k$--Yamabe soliton is also constant. Hence, in such a case Theorem~\ref{teo1} does not apply. In the complete, noncompact case, as an immediate application of
Theorem~\ref{teo1}, we can prove the following global result. 

\begin{teo}\label{teokY} Let $(M^{n},g)$ be a complete, noncompact,
  gradient $k$--Yamabe soliton with nonnegative Ricci tensor and
  nonconstant potential function $f$. Then, we have the following two cases:
\begin{itemize}  
 \item[(1)] either $(M^{n},g)$ is a direct product
  $\RR\times N^{n-1}$ where $(N^{n-1},g^N)$ is 
an $(n-1)$--dimensional complete Riemannian manifold with nonnegative Ricci tensor. If in addition $(M^{n},g)$ is locally conformally flat, then either it is flat or 
the manifold $(N^{n-1},g^N)$ is a quotient of the round sphere $\SS^{n-1}$;
\smallskip
\item[(2)] or $(M^{n},g)$ is rotationally symmetric and globally
  conformally equivalent to $\RR^n$. More precisely, there exists a
  point $O\in M^n$ such that on $M^{n}\setminus \{O\}$ 
  the metric has the form
$$
g \, = \, v^{2}(t)\, \big(dt^{2}+t^{2}g^{\SS^{n-1}}\big) \,,
$$
where $v:\RR^+\to \RR$ is some positive smooth function.
\end{itemize}
\end{teo}

From Remark~\ref{remark}, it is now easy to deduce the following corollary.
\begin{cor} Let $(M^{n},g)$ be a complete, noncompact, gradient
  $k$--Yamabe soliton with nonnegative Ricci tensor and nonconstant
  potential function $f$. If the Ricci tensor is positive at some
  point, then $(M^{n},g)$ is rotationally symmetric and globally
  conformally equivalent to $\RR^n$.
\end{cor}

\

\begin{ackn} 
The authors are partially supported by the Italian project FIRB--IDEAS ``Analysis and Beyond''.
\end{ackn}

\medskip

\noindent {\bf Note.} {\em During the editing of this work, H.-D.~Cao,
  X.~Sun and Y.~Zhang posted on the {\em ArXiv Preprint Server} the
  manuscript~\cite{caosunzhang}, where a classification result for
  gradient Yamabe solitons similar to the one discussed in
  Section~\ref{s_yamabe} (in particular, Theorem~\ref{teoY}) is
  obtained.}

\

\

\bibliographystyle{amsplain}
\bibliography{confsol}

\

\end{document}